\documentclass[12pt]{amsart}
\usepackage[alphabetic]{amsrefs}
\usepackage{mathdots, blkarray}
\usepackage{multicol}
\usepackage{graphicx}
\usepackage{amscd}
\usepackage{upgreek}
\usepackage{stmaryrd}
\usepackage{longtable}
\usepackage[T1]{fontenc}
\usepackage{latexsym, amsmath, amssymb, amsthm}
\usepackage[svgnames]{xcolor}
\usepackage{rsfso}
\usepackage[mathscr]{eucal}
\usepackage{mathtools}
\usepackage{mathptmx}
\usepackage{titletoc}
\usepackage{wrapfig}
\usepackage{float}
\usepackage{xypic}
\usepackage{microtype}
\usepackage{dsfont}
\usepackage{xcolor}
\usepackage{color}
\usepackage[colorlinks = true,
            linkcolor  = DarkBlue,
            urlcolor   = DarkRed,
            citecolor  = DarkGreen]{hyperref}
\usepackage{hyperref}
\allowdisplaybreaks	
\usepackage[centering, includeheadfoot, hmargin=1.0in, tmargin=0.5in, 
  bmargin=0.6in, headheight=6pt]{geometry}
\newtheorem{theorem}{Theorem}[section]
\newtheorem{prop}[theorem]{Proposition}

\newtheorem{lem}[theorem]{Lemma}
\newtheorem{coro}[theorem]{Corollary}

\newtheorem{thm}[theorem]{Theorem}

\newtheorem{rem}[theorem]{\rm\textsc{Remark}}

\newcommand{\bslash}{\kern-0.1em\texttt{\scalebox{0.6}[1]{/}}\kern-0.15em \texttt{\scalebox{0.6}[1]{/}}}

\DeclareMathOperator{\GL}{GL}
\DeclareMathOperator{\SL}{SL}

\DeclareMathOperator{\rank}{rank}

\DeclareMathOperator{\orbit}{orbit}

\DeclareMathOperator{\height}{height}
\newcommand{\A}{\mathcal{A}}

\newcommand{\N}{\mathbb{N}} 
\newcommand{\F}{\mathbb{F}}

\newcommand{\ra}{\longrightarrow}

\newcommand{\hbo}{$\hfill\Diamond$}

\begin{document}
\title{Modular matrix invariants under some transpose actions} 
\def\shorttitle{Modular matrix invariants under some transpose actions}

\author{Yin Chen}
\address{School of Mathematics and Physics, Key Laboratory of ECOFM of 
Jiangxi Education Institute, Jinggangshan University,
Ji'an 343009, Jiangxi, China \& Department of Finance and Management Science, University of Saskatchewan, Saskatoon, SK, Canada, S7N 5A7}
\email{yin.chen@usask.ca}

\author{Shan Ren}
\address{School of Mathematics and Statistics, Northeast Normal University, Changchun 130024, China}
\email{rens734@nenu.edu.cn}

\begin{abstract}
Consider the special linear group of degree $2$ over an arbitrary finite field, acting on the full space of $2 \times 2$-matrices by transpose. We explicitly construct a generating set for the corresponding modular matrix invariant ring, demonstrating that this ring is a hypersurface. Using a recent result on $a$-invariants of Cohen-Macaulay algebras, we determine the Hilbert series of this invariant ring, and our method avoids seeking the generating relation. Additionally, we prove that the modular matrix invariant ring of the group of upper triangular $2 \times 2$-matrices is also a hypersurface.
\end{abstract}

\date{\today}
%\subjclass[2010]{13A50.}
\maketitle \baselineskip=16.3pt

%%%%%%%%%%%%%%%%%%%%%%%%%%%Contents%%%%%%%%%%%%%%%%%%%%%%%%
%\textcolor{blue}{\tableofcontents{}}
\dottedcontents{section}[1.16cm]{}{1.8em}{5pt}
\dottedcontents{subsection}[2.00cm]{}{2.7em}{5pt}
%\dottedcontents{subsubsection}[2.86cm]{}{3.4em}{5pt}

%%%%%%%%%%%%%%%%%%%%%%%%%%%Sections%%%%%%%%%%%%%%%%%%%%%%%%
\section{Introduction}
\setcounter{equation}{0}
\renewcommand{\theequation}
{1.\arabic{equation}}
\setcounter{theorem}{0}
\renewcommand{\thetheorem}
{1.\arabic{theorem}}

\noindent Let $\F_q$ be a finite field of order $q=p^s$ ($s\in\N^+$) and $\GL_n(\F_q)$ be the general linear group of degree $n$ over $\F_q$. Suppose that $G\leqslant \GL_n(\F_q)$ is a subgroup and $M_n(\F_q)$ denotes the space of all $n\times n$ matrices over $\F_q$. Consider the following transpose action of $G$ on $M_n(\F_q)$ defined by
\begin{equation}
\label{trans}
(g,M)\mapsto g\cdot M\cdot g^{t}
\end{equation}
for all $g\in G$ and $M\in M_n(\F_q)$. This action indicates that $M_n(\F_q)$ can be viewed as an $n^2$-dimensional (not necessarily faithful) representation of $G$ over $\F_q$. The main objective of this article is to calculate  
the corresponding modular invariant rings $\F_q[M_2(\F_q)]^G,$
where $G$ denotes the group  $U_2(\F_q)$ of upper triangular matrices of size $2$ and the special linear group $\SL_2(\F_q)$; see for example, \cite{CW11} or \cite{DK15} for general references about modular invariant theory of finite groups.

This study is motivated by the following aspects. Firstly, a fundamental philosophy  in modular invariant theory is to extend key results (or objects) from characteristic zero to the modular setting. For example, the well-studied invariant theory of vectors and covectors in characteristic zero has recently been generalized to the modular case; see \cite{BK11, Che14, Che21, CT19, Ren24}, and \cite{RZ25}. While the matrix invariant theory in characteristic zero is well-developed (\cite{DP17}), research on modular matrix invariants over finite fields remains in its early stages, with very few results available, even for the case $n=2$.
Secondly, for the small prime field $\F_2$, Smith and Stong \cite[Corollary 3.2]{SS10a} have computed the invariant ring $\F_2[M_2(\F_2)]^{\GL_2(\F_2)}$ and showed that it is a hypersurface, 
where the action of $\GL_2(\F_2)$ on $M_2(\F_2)$ is given by the transpose action (\ref{trans}). This work also influences Smith and Stong's subsequent work about Poincar\'{e} duality algebras mod $2$; see  \cite{SS10b}. This means that extending the characteristic 2 case to arbitrary prime characteristics, may lead to more potential applications in algebraic topology. Thirdly, the $a$-invariant theory of invariant rings, which is useful for determining the Hilbert series of an invariant ring, has made some progress in \cite[Theorem 4.4]{GJS25}, revealing when an invariant ring and its underlying polynomial ring share the same $a$-invariants. Recently, this result has been used by \cite{Mai25} to compute the modular matrix invariants of $\GL_2(\F_q)$ acting on $M_2(\F_q)$ by conjugation. Thus, it is natural to explore whether this result can be applied to the calculation of modular matrix invariants under some transpose actions.

To articulate our methods and results, we choose a basis of $M_2(\F_q)$ as follows:
\begin{equation}
\label{eq1.2}
\left\{e_1=\begin{pmatrix}
      1&0    \\
     0 &0  
\end{pmatrix},e_2=\begin{pmatrix}
      0&0    \\
     1 &0  
\end{pmatrix}, e_3=\begin{pmatrix}
      0&1    \\
     0 &0  
\end{pmatrix}, e_4=\begin{pmatrix}
      0&0    \\
     0 &1 
\end{pmatrix}\right\}.
\end{equation}
Let $M_2^*(\F_q)$ denote the dual space of $M_2(\F_q)$ and choose a basis $\{x_1,x_2,x_3,x_4\}$ of $M_2^*(\F_q)$, dual to $\{e_1,e_2,e_3,e_4\}$. We may identify the symmetric algebra $\F_q[M_2(\F_q)]$ on $M_2^*(\F_q)$ with
the polynomial ring $\F_q[x_1,x_2,x_3,x_4]$.  For any $G\leqslant\GL_2(\F_q)$, the matrix invariant ring 
$\F_q[M_2(\F_q)]^G$ denotes the subring of $\F_q[M_2(\F_q)]$ consisting of all polynomials fixed under the action of $G$.
The ring $\F_q[M_2(\F_q)]^G$ has Krull dimension 4, meaning its generating set requires at least 4 elements. 

In Section \ref{sec2}, we compute the invariant ring $\F_q[M_2(\F_q)]^{U_2(\F_q)}$ and determine its algebraic structure and algebraic properties. Note that $U_2(\F_q)$ is of order $q$ and actually it is isomorphic to a Sylow $p$-subgroup of any group $H$ between $\SL_2(\F_q)$ and $\GL_2(\F_q)$, we see that $\F_q[M_2(\F_q)]^{U_2(\F_q)}$ is factorial and the Cohen-Macaulayness of $\F_q[M_2(\F_q)]^{U_2(\F_q)}$ (Corollary \ref{coro1}) implies that $\F_q[M_2(\F_q)]^{H}$ is also Cohen-Macaulay. 
To find a generating set of $\F_q[M_2(\F_q)]^{U_2(\F_q)}$, we construct five $U_2(\F_q)$-invariants $\{f_1,f_2,f_3,f_4,\upzeta\}$ and show that there exists a subgroup $\widetilde{U}_2(\F_q)$ of $\GL_4(\F_q)$ containing  the image of $U_2(\F_q)$ in $\GL_4(\F_q)$ such that
$[\widetilde{U}_2(\F_q): U_2(\F_q)]=2$ and 
\begin{equation}
\label{ }
\F_q[M_2(\F_q)]^{\widetilde{U}_2(\F_q)}=\F_q[f_1,f_2,f_3,f_4]
\end{equation}
 is a polynomial algebra (Proposition \ref{prop2.2}). This means that 
$\F_q[M_2(\F_q)]^{U_2(\F_q)}$ is a free module of rank $2$ over $\F_q[M_2(\F_q)]^{\widetilde{U}_2(\F_q)}$. 
The  last new invariant $\upzeta$ can be determined by computing the Hilbert series of $\F_q[M_2(\F_q)]^{U_2(\F_q)}$ and applying \cite[Theorem 4.4]{GJS25}. The main results of this section are Theorem \ref{thm1} and Corollary \ref{coro1}, which demonstrate that 
\begin{equation}
\label{ }
\F_q[M_2(\F_q)]^{U_2(\F_q)}=\F_q[f_1,f_2,f_3,f_4,\upzeta]
\end{equation}
is also a hypersurface.
The advantage of our method is that it avoids the need to find the generating relation among $\{f_1,f_2,f_3,f_4,\upzeta\}$.

Section \ref{sec3} is devoted to computing the invariant ring $\F_q[M_2(\F_q)]^{\SL_2(\F_q)}$ for $p\geqslant 3$ and $q\neq 3$, with a remark about the case $q=3$ (Remark \ref{rem3.4}) and a discussion of the case $p=2$ at the end.
We observe that the two $U_2(\F_q)$-invariants $f_2$ and $f_3$ are also fixed by the action of $\SL_2(\F_q)$. We construct three new  $\SL_2(\F_q)$-invariants: $g_0, g_1, g_2$. Our construction method is inspired by \cite[Section 2]{HS11}. The same procedure outlined in Section \ref{sec2} can be applied to $\F_q[M_2(\F_q)]^{\SL_2(\F_q)}$.
We eventually prove in Theorem \ref{thm2} that 
\begin{equation}
\label{ }
\F_q[M_2(\F_q)]^{\SL_2(\F_q)}=\F_q[f_2,f_3,g_0,g_1,g_2]
\end{equation}
 is a hypersurface as well.

It is worth noting that MAGMA calculations \cite{BCP97} report that the cardinality of a generating set of $\F_q[M_2(\F_q)]^{\GL_2(\F_q)}$ 
will increase as $q$ increases for the transpose action, while \cite[Corollary 5.6]{Mai25} asserts that the invariant ring of  $\GL_2(\F_q)$ acting on 
$M_2(\F_q)$ by conjugation can be generated by five invariants. Moreover,
\cite{Ang98} and \cite{Smi02} studied modular matrix invariants of degree 2 under the conjugation action and see \cite{CZ23} for an application of modular conjugate matrix invariants in understanding the geometry of matrix equations over finite fields. Some works also explore hypersurfaces in invariant theory in a more conceptual manner; see \cite{Nak83, CH99, HK05} and \cite{Bro06}.

\section{$U_2(\F_q)$-Invariants} \label{sec2}
\setcounter{equation}{0}
\renewcommand{\theequation}
{2.\arabic{equation}}
\setcounter{theorem}{0}
\renewcommand{\thetheorem}
{2.\arabic{theorem}}

\noindent For any element $c\in\F_q$, we define $\updelta_c:=\begin{pmatrix}
    1  &   c \\
     0 &  1
\end{pmatrix}\in U_2(\F_q)$ and suppose $
M:=\begin{pmatrix}
     a_1 & a_3   \\
     a_2 & a_4 
\end{pmatrix}\in M_2(\F_q)
$
denotes a generic matrix.  Then
\begin{equation}
\label{ }
\updelta_c\cdot M\cdot \updelta_c^{t}=\begin{pmatrix}
    1  &   c \\
     0 &  1
\end{pmatrix}\begin{pmatrix}
     a_1 & a_3   \\
     a_2 & a_4 
\end{pmatrix}\begin{pmatrix}
    1  &   0 \\
     c &  1
\end{pmatrix}=\begin{pmatrix}
  a_1+ca_2+ca_3+c^2a_4    &  a_3+ca_4  \\
   a_2+ca_4   &a_4  
\end{pmatrix}.
\end{equation}
The resulting matrix $[\updelta_c]$ of $\updelta_c$ on the vector space $M_2(\F_q)$ with respect to $\{e_1,e_2,e_3,e_4\}$ is
\begin{equation}
\label{ }
[\updelta_c]=\begin{pmatrix}
    1  &  0 & 0&0\\
    c  &1 &0&0\\
    c  &0& 1&0\\
    c^2  &c& c&1\\
\end{pmatrix}.
\end{equation}
Thus with respect to the dual basis $\{x_1,x_2,x_3,x_4\}$, the resulting matrix of $\updelta_c$ on $M_2^*(\F_q)$  is
\begin{equation}
\label{ }
([\updelta_c]^t)^{-1}=\begin{pmatrix}
    1  &  -c & -c&c^2\\
    0  &1 &0 &-c\\
    0  &0& 1&-c\\
    0  &0&0&1\\
\end{pmatrix}.
\end{equation}
Note that $U_2(\F_q)=\{\updelta_c\mid c\in\F_q\}$. Hence, the action of $U_2(\F_q)$ on $\F_q[M_2(\F_q)]=\F_q[x_1,x_2,x_3,x_4]$ can be given by 
\begin{equation} 
\begin{aligned}
\updelta_c(x_1) &=& x_1 \hspace{1cm}&& \updelta_c(x_2) &=& x_2-cx_1\hspace{2.15cm} \\
\updelta_c(x_3) &=& x_3-cx_1 && \updelta_c(x_4) &=& x_4-cx_3-cx_2+c^2x_1
\end{aligned}
\end{equation}
and clearly, we see that this action is faithful.

Given a polynomial $f\in \F_q[M_2(\F_q)]$, we write $\orbit(f)$ for the orbit of $f$ under the action of $U_2(\F_q)$.
It can be directly verified that the following polynomials are $U_2(\F_q)$-invariants:
$$f_1:= x_1 \hspace{1cm}
f_2:= x_2-x_3\hspace{1cm}
f_3:=x_1x_4-x_2x_3$$
and 
\begin{equation} 
\begin{aligned}
f_4&:=\prod_{\ell\in\orbit(x_4)}\hspace{-5mm}\ell ~=\prod_{c\in\F_q}(x_4-cx_3-cx_2+c^2x_1)\\
\upzeta&:=\prod_{\ell\in\orbit(x_3)} \hspace{-2mm}\ell ~=\prod_{c\in\F_q}(x_3-cx_1)=x_3^q-x_1^{q-1}x_3
\end{aligned}
\end{equation}
Note that $f_3$ is the determinant function exactly. Thus it is also an $\SL_2(\F_q)$-invariant.
Moreover, we define the following family of $U_2(\F_q)$-invariants:
\begin{equation}
\label{ }
h_a:=\prod_{\ell\in\orbit(x_4-ax_1)}\hspace{-5mm}\ell ~=\prod_{c\in\F_q}(x_4-cx_3-cx_2+(c^2-a)x_1)
\end{equation}
for all $a\in\F_q$. Apparently, $h_0=f_4$.  These polynomials $h_a$ will be used to construct new $\SL_2(\F_q)$-invariants in Section \ref{sec3}.

\begin{lem}\label{lem2.1}
The set $\{f_1,f_2,f_3,f_4\}$ is a homogeneous system of parameters.
\end{lem}

\begin{proof}
We will use \cite[Lemma 2.6.3]{CW11} to prove this statement. Let $\overline{V}:=M_2(\F_q)\otimes_{\F_q}\overline{\F}_q$ be the $4$-dimensional vector space over an algebraic closure $\overline{\F}_q$ of $\F_q$ and suppose $v=(a_1,a_2,a_3,a_4)\in \overline{V}$ denotes an arbitrary point in the algebraic set defined by $\{f_1,f_2,f_3,f_4\}$. Since $f_1(v)=f_2(v)=0$, it follows that
$a_1=0$ and $a_2=a_3$. Together with $f_3(v)=0$, we see that $a_2=a_3=0$. Note that
$$0=f_4(v)=\prod_{c\in\F_q}(a_4-ca_3-ca_2+c^2a_1)=\prod_{c\in\F_q}a_4=a_4^q.$$
Thus $a_4=0$ and hence, $v=0$. Now \cite[Lemma 2.6.3]{CW11} applies.
\end{proof}

We will see that the polynomial algebra $\F_q[f_1,f_2,f_3,f_4]$ is actually the invariant ring of a group 
$\widetilde{U}_2(\F_q)$ that contains $U_2(\F_q)$ as a subgroup. 

Define
\begin{equation}
\label{ }
\upalpha:=\begin{pmatrix}
    1  & 0&0&0   \\
  0  & 0&-1&0   \\
    0  & -1&0&0   \\
     0   & 0&0&1   \\
\end{pmatrix}\in \GL_4(\F_q)
\end{equation}
and define $\widetilde{U}_2(\F_q)$ as the subgroup of $\GL_4(\F_q)$ generated by the image of $U_2(\F_q)$ on $M_2(\F_q)$ and $\upalpha$.
It is clear that the order of $\upalpha$ is 2 and $|\widetilde{U}_2(\F_q)|=2q$.

\begin{prop}\label{prop2.2}
$\F_q[M_2(\F_q)]^{\widetilde{U}_2(\F_q)}=\F_q[f_1,f_2,f_3,f_4]$.
\end{prop}

\begin{proof}
First of all, we note that $\upalpha$ fixes $x_1$ and $x_4$, and $\upalpha(x_2)=-x_3$ and $\upalpha(x_3)=-x_2$.
Using this action, one checks directly that $f_1,f_2,$ and $f_3$ are $\widetilde{U}_2(\F_q)$-invariants. To see
$f_4$ is also fixed by $\widetilde{U}_2(\F_q)$, it suffices to verify that $\upalpha(f_4)=f_4$. In fact,
\begin{eqnarray*}
\upalpha(f_4)&=&\upalpha\left(\prod_{c\in\F_q}(x_4-cx_3-cx_2+c^2x_1)\right)\\
&=&\prod_{c\in\F_q}(x_4+cx_2+cx_3+c^2x_1)=\prod_{c\in\F_q}(x_4-(-c)x_2-(-c)x_3+(-c)^2x_1)\\
&=&\prod_{-c\in\F_q}(x_4-(-c)x_2-(-c)x_3+(-c)^2x_1)=f_4.
\end{eqnarray*}
Hence, $\F_q[f_1,f_2,f_3,f_4]\subseteq \F_q[M_2(\F_q)]^{\widetilde{U}_2(\F_q)}$. 
Moreover, note that
$$\deg(f_1)\cdot \deg(f_2)\cdot \deg(f_3)\cdot \deg(f_4)=2q=|\widetilde{U}_2(\F_q)|$$
and $\{f_1,f_2,f_3,f_4\}$ is a homogeneous system of parameters by Lemma \ref{lem2.1}. Thus 
it follows from \cite[Proposition 16]{Kem96} (or \cite[Corollary 3.1.6]{CW11}) that $\F_q[M_2(\F_q)]^{\widetilde{U}_2(\F_q)}=\F_q[f_1,f_2,f_3,f_4]$ is a polynomial algebra over $\F_q$.
\end{proof}

\begin{rem}\label{rem2.3}{\rm
Note that $\upzeta$ is not fixed under the action of $\widetilde{U}_2(\F_q)$ because $\upalpha(\upzeta)\neq \upzeta$.
\hbo}\end{rem}

We are ready to prove the main result in this section. 

\begin{thm}\label{thm1}
$\F_q[M_2(\F_q)]^{U_2(\F_q)}=\F_q[f_1,f_2,f_3,f_4,\upzeta]$. 
\end{thm}

\begin{proof}
The first fact we need to prove is that $\F_q[M_2(\F_q)]^{U_2(\F_q)}$ is Cohen-Macaulay. Note that for $c\in\F_q^{\times}$,
$\rank([\updelta_c]-I_4)=2$. Thus the fixed subspace of $U_2(\F_q)$ on $M_2(\F_q)$ has dimension $2$. By \cite[Corollary 3.9.2]{CW11}, it follows that $\F_q[M_2(\F_q)]^{U_2(\F_q)}$ is Cohen-Macaulay.

By Lemma \ref{lem2.1} and \cite[Corollary 3.0.6]{CW11}, it follows that $\{f_1,f_2,f_3,f_4\}$ is a homogeneous system of parameters for $\F_q[M_2(\F_q)]^{U_2(\F_q)}$. Thus $\F_q[M_2(\F_q)]^{\widetilde{U}_2(\F_q)}=\F_q[f_1,f_2,f_3,f_4]$ is a Noether normalization of $\F_q[M_2(\F_q)]^{U_2(\F_q)}$ by Proposition \ref{prop2.2}. 
Namely, $\F_q[M_2(\F_q)]^{U_2(\F_q)}$ can be viewed as a free $\F_q[M_2(\F_q)]^{\widetilde{U}_2(\F_q)}$-module. 
It follows from \cite[Corollary 3.1.4]{CW11} that the rank of this free module is
$$\frac{|\widetilde{U}_2(\F_q)|}{|U_2(\F_q)|}=\frac{2q}{q}=2.$$
Thus, we may write the corresponding Hironaka decomposition of $\F_q[M_2(\F_q)]^{U_2(\F_q)}$ as
\begin{equation}
\label{ }
\F_q[M_2(\F_q)]^{U_2(\F_q)}=\F_q[M_2(\F_q)]^{\widetilde{U}_2(\F_q)}\oplus\F_q[M_2(\F_q)]^{\widetilde{U}_2(\F_q)}\cdot  s
\end{equation}
where $s\notin \F_q[M_2(\F_q)]^{\widetilde{U}_2(\F_q)}$ denotes a homogeneous $U_2(\F_q)$-invariant. Then the Hilbert series of $\F_q[M_2(\F_q)]^{U_2(\F_q)}$ can be written as
\begin{equation}
\label{ }
\mathcal{H}\left(\F_q[M_2(\F_q)]^{U_2(\F_q)}, \uplambda\right)=\frac{1+\uplambda^{\deg(s)}}{(1-\uplambda)(1-\uplambda)(1-\uplambda^2)(1-\uplambda^q)}.
\end{equation}
Since $\F_q[M_2(\F_q)]^{U_2(\F_q)}$ is a graded Cohen-Macaulay algebra, by a theorem of Serre (see \cite[Theorem 4.4.3]{BH93}), we see that the $a$-invariant of $\F_q[M_2(\F_q)]^{U_2(\F_q)}$ is equal to $\deg(s)-(q+4)$. On the other hand,
the fact that $\rank([\updelta_c]-I_4)\leqslant 1$ if and only if $c=0$ in $\F_q$, implies that the image of $U_2(\F_q)$
on $M_2(\F_q)$ contains no reflections. By \cite[Theorem 4.4]{GJS25}, we see that the $a$-invariants of $\F_q[M_2(\F_q)]^{U_2(\F_q)}$ and $\F_q[M_2(\F_q)]$ coincides. Note that the latter is $-4$. Hence,
$$\deg(s)-(q+4)=-4$$
which implies that $\deg(s)=q$ and so the Hilbert series of $\F_q[M_2(\F_q)]^{U_2(\F_q)}$ is
$$\mathcal{H}\left(\F_q[M_2(\F_q)]^{U_2(\F_q)}, \uplambda\right)=\frac{1+\uplambda^{q}}{(1-\uplambda)(1-\uplambda)(1-\uplambda^2)(1-\uplambda^q)}.$$

Let $N$ be the submodule of $\F_q[M_2(\F_q)]^{U_2(\F_q)}$ generated by $\{1,\upzeta\}$ over $\F_q[M_2(\F_q)]^{\widetilde{U}_2(\F_q)}$. We observe that $\{1,\upzeta\}$ is linearly independent over $\F_q[M_2(\F_q)]^{\widetilde{U}_2(\F_q)}$  because if there were two elements $\upalpha_1,\upalpha_2\in \F_q[M_2(\F_q)]^{\widetilde{U}_2(\F_q)}$, not both zero, such that
$\upalpha_1+\upalpha_2\cdot \upzeta=0$, then $\upzeta=-\frac{\upalpha_1}{\upalpha_2}$ will be an $\widetilde{U}_2(\F_q)$-invariant.  This contradicts with Remark \ref{rem2.3}. Thus, $N$ is a free module and we may write 
$$N=\F_q[M_2(\F_q)]^{\widetilde{U}_2(\F_q)}\oplus\F_q[M_2(\F_q)]^{\widetilde{U}_2(\F_q)}\cdot  \upzeta.$$
Clearly, 
$$\mathcal{H}\left(N, \uplambda\right)=\frac{1+\uplambda^{q}}{(1-\uplambda)(1-\uplambda)(1-\uplambda^2)(1-\uplambda^q)}=\mathcal{H}\left(\F_q[M_2(\F_q)]^{U_2(\F_q)}, \uplambda\right).$$
Therefore, 
$\F_q[M_2(\F_q)]^{U_2(\F_q)}=\F_q[M_2(\F_q)]^{\widetilde{U}_2(\F_q)}\oplus\F_q[M_2(\F_q)]^{\widetilde{U}_2(\F_q)}\cdot  \upzeta=\F_q[f_1,f_2,f_3,f_4,\upzeta].$
\end{proof}

\begin{coro}\label{coro1}
$\F_q[M_2(\F_q)]^{U_2(\F_q)}$ is a hypersurface. 
\end{coro}

\begin{proof} 
Let $S:=\F_q[X_1,\dots,X_4, Y]$ denote a polynomial ring in five variables over $\F_q$ and consider the standard $\F_q$-algebra homomorphism $\uppi: S\ra \F_q[M_2(\F_q)]^{U_2(\F_q)}$ defined by sending $Y$ to $\upzeta$ and $X_i$ to $f_i$  for $i=1,\dots,4$, respectively. By Theorem \ref{thm1}, we see that $\uppi$ is surjective. Hence, 
$$S/\ker(\uppi)\cong \F_q[M_2(\F_q)]^{U_2(\F_q)}$$
in which the latter is an integral domain, so $\ker(\uppi)$ is a prime ideal of $S$. 
As $S$ is an integral domain of Krull dimension $5$ and the Krull dimension of $S/\ker(\uppi)$ is $4$, it follows from \cite[Corollary 13.4]{Eis95} that $\height(\ker(\uppi))=\dim(S)-\dim(S/\ker(\uppi))=5-4=1.$
Since $S$ is Noetherian and factorial, it follows from \cite[Theorem 20.1]{Mat89} that
$\ker(\uppi)$ is principal, i.e., generated by an irreducible element $Q$, we say. This also means that
$\F_q[M_2(\F_q)]^{U_2(\F_q)}$ is generated by $\{f_1,f_2,f_3,f_4,\upzeta\}$, subject to
the unique generating relation $Q(f_1,f_2,f_3,f_4,\upzeta)=0$. Therefore, it is a hypersurface. 
\end{proof}

\section{$\SL_2(\F_q)$-Invariants} \label{sec3}
\setcounter{equation}{0}
\renewcommand{\theequation}
{3.\arabic{equation}}
\setcounter{theorem}{0}
\renewcommand{\thetheorem}
{3.\arabic{theorem}}

\noindent Throughout this section we suppose $p\geqslant 3$ and $q\neq 3$.  Recall that the group $U_2(\F_q)$ can be viewed as a Sylow $p$-subgroup of $\SL_2(\F_q)$ and $\SL_2(\F_q)$ can be generated by $U_2(\F_q)$ and the following element 
$$\uptau:=\begin{pmatrix}
    0  &   1 \\
     -1 &  0
\end{pmatrix}.$$
Note that the resulting matrix of $\uptau$ on the vector space $M_2(\F_q)$ is
\begin{equation}
\label{ }
[\uptau]=\begin{pmatrix}
    0  &  0 & 0&1\\
    0  &0 &-1 &0\\
    0  &-1& 0&0\\
    1  &0&0&0\\
\end{pmatrix}
\end{equation}
and $[\uptau]=([\uptau]^t)^{-1}.$ Thus 
\begin{equation}
\label{ }
\uptau(x_1)=x_4,\uptau(x_4)=x_1,\uptau(x_2)=-x_3,\textrm{ and }\uptau(x_3)=-x_2.
\end{equation}

We observe that the group homomorphism $\upvarphi:\SL_2(\F_q)\ra\GL_4(\F_q)$ induced by the action of $\SL_2(\F_q)$ on $M_2(\F_q)$ is not injective. In fact, the kernel of $\upvarphi$ is generated by $-I_2$ (the negative identity matrix),
and so $|\ker(\upvarphi)|=2$. This also implies that the order of the image of $\SL_2(\F_q)$ on $M_2(\F_q)$ is
$$|\upvarphi(\SL_2(\F_q))|=\frac{|\SL_2(\F_q)|}{2}=\frac{q^3-q}{2}.$$

We write $\A_0$ for the set of quadratic residues in $\F_q$ (i.e., $\A_0=\{c^2\mid c\in\F_q^\times\}$) and denote by $\A_1$  the set of quadratic non-residues (i.e., $\A_1=\F_q^\times\setminus\A_0$). Thus
$|\A_0|=|\A_1|=\frac{q-1}{2}$.

Consider a new $U_2(\F_q)$-invariant of degree $q$:
$$f_0:=\prod_{\ell\in\orbit(x_2+x_3)}\hspace{-3mm}\ell ~~=\prod_{c\in\F_q}(x_3+x_2-2cx_1)=(x_3+x_2)^q-(x_3+x_2)x_1^{q-1}$$
and two known $U_2(\F_q)$-invariants $f_2 = x_2-x_3$ and $f_3=x_1x_4-x_2x_3$ already appeared in  Section \ref{sec2}. Let us define the following three new polynomials:
\begin{equation} 
\begin{aligned}
g_0&:=x_1^qx_4 + x_1x_4^q - x_2^qx_3 -x_2x_3^q\\
g_1&:=\prod_{a\in\A_1}h_a\\
g_2&:=f_0\cdot\prod_{a\in\A_0}h_a.
\end{aligned}
\end{equation}

These five polynomials above are all $\SL_2(\F_q)$-invariants.

\begin{lem}\label{lem3.1}
$f_2,f_3,g_0,g_1,g_2\in\F_q[M_2(\F_q)]^{\SL_2(\F_q)}$.
\end{lem}

\begin{proof}
Note that $f_2$ and $f_3$ are $U_2(\F_q)$-invariants, and clearly, $\uptau(f_2)=f_2$ and $\uptau(f_3)=f_3$. Thus $f_2$ and $f_3$ are also $\SL_2(\F_q)$-invariants. To prove the $\SL_2(\F_q)$-invariance of $g_0$, it is easy to see that $\uptau(g_0)=g_0$, thus it suffices to verify that $\updelta_c(g_0)=g_0$ for all $c\in\F_q$. In fact,
\begin{eqnarray*}
\updelta_c(g_0)& = & \updelta_c(x_1^qx_4 + x_1x_4^q - x_2^qx_3 -x_2x_3^q) \\
 & = & x_1^q(x_4-cx_3-cx_2+c^2x_1) + x_1(x_4-cx_3-cx_2+c^2x_1)^q \\
 &&- (x_2-cx_1)^q(x_3-cx_1) -(x_2-cx_1)(x_3-cx_1)^q\\
 &=&x_1^qx_4 + x_1x_4^q - x_2^qx_3 -x_2x_3^q=g_0.
\end{eqnarray*}
Hence, $g_0$ is also fixed by $\SL_2(\F_q)$.

Now let us show that $g_1$ and $g_2$ are $\SL_2(\F_q)$-invariants as well. We first \textit{claim} that 
$g_1$ and $g_2$ both are projective orbit products of $\SL_2(\F_q)$ acting on the set of lines in the dual of $M_2(\F_q)$. 
Hence, for $i\in\{1,2\}$, there exists $b_i\in\F_q^\times$ such that
$$\uptau(g_i) =b_i\cdot g_i.$$
Note that $g_i$ is fixed by $U_2(\F_q)$ as each $h_a$ is $U_2(\F_q)$-invariant, thus the action of $\SL_2(\F_q)$ on $g_i$  induces a multiplicative character of $\SL_2(\F_q)$ over $\F_q$. However, for $q\notin\{2,3\}$, it is well-known that $\SL_2(\F_q)$ is a simple group, thus this character must be trivial, i.e., $b_i=1$. This also proves that $g_1$ and $g_2$ are $\SL_2(\F_q)$-invariants.

Hence, it suffices to prove the claim above.  Let us define 
\begin{equation}
\label{ }
\ell_a(c):=x_4-cx_3-cx_2+(c^2-a)x_1
\end{equation}
and consider the following subset of the dual space of $M_2(\F_q)$:
$$\Delta_1:=\left\{\ell_a(c)\mid a\in\A_1,c\in \F_q\right\}.$$
Apparently, $|\Delta_1|=\frac{q(q-1)}{2}$ because there exists a bijective map between $\A_1\times \F_q$ and $\Delta_1$. Note that for each $\ell_a(c)\in \Delta_1$, since $a\in\A_1$, the coefficient $c^2-a$ of $x_1$ in $\ell_a(c)$ is nonzero. Moreover, 
\begin{eqnarray*}
\uptau(\ell_a(c)) & = &\uptau\left(x_4-cx_3-cx_2+(c^2-a)x_1\right) =x_1+cx_2+cx_3+(c^2-a)x_4 \\ 
&=&(c^2-a)\left(x_4-\frac{-c}{c^2-a}x_3-\frac{-c}{c^2-a}x_2+\frac{1}{c^2-a} x_1\right)\\
&=&(c^2-a)\left(x_4-\frac{-c}{c^2-a}x_3-\frac{-c}{c^2-a}x_2+\left(\left(\frac{-c}{c^2-a}\right)^2-\frac{a}{(c^2-a)^2}\right) x_1\right)\\
&=&(c^2-a)\cdot \ell_{b}(d)
\end{eqnarray*}
where $b:=\frac{a}{(c^2-a)^2}\in\A_1$ (because of $a\in\A_1$) and $d:=\frac{-c}{c^2-a}$.
This observation, together with the fact that $\ell_a(c)=\updelta_c(x_4-ax_1)=\updelta_c(\ell_a(0))$, implies that
$\SL_2(\F_q)$ permutes $\Delta_1$, up to nonzero scalars. In terms of projective geometry, this means that the set of projective points corresponding to $\Delta_1$ is an orbit of the projective points corresponding to $\{\ell_a(0)\mid a\in\A_1\}$ under the action of $\SL_2(\F_q)$. Hence, $g_1=\prod_{a\in\A_1}h_a=\prod_{\ell\in \Delta_1}\ell$
is a projective orbit product of $\SL_2(\F_q)$.

Similarly, we define 
$$\Delta_0:=\{x_3+x_2-cx_1\mid c\in\F_q\}\cup\left\{\ell_a(c)\mid a\in\A_0,c\in \F_q\right\}.$$
Clearly, $|\Delta_0|=q+\frac{q(q-1)}{2}=\frac{q(q+1)}{2}$ and 
$g_2=\prod_{\ell\in \Delta_0}\ell.$ To show $g_2$ is a projective orbit product of $\SL_2(\F_q)$, we only need to show that
the action  of $\SL_2(\F_q)$ on $\Delta_0$ is projectively stable (i.e., $\SL_2(\F_q)$ permutes $\Delta_0$ up to nonzero scalars). We denote by $[\ell]$ the projective point corresponding to a linear function $\ell$ and write $[\Delta_0]$ for
$\{[\ell]\mid \ell\in \Delta_0\}$. Recall that $\{x_3+x_2-cx_1\mid c\in\F_q\}$ is the orbit of $x_3+x_2$ under the action of $U_2(\F_q)$ and  $\uptau(x_3+x_2)=-(x_3+x_2)$, thus the orbit of $[x_3+x_2]$ is contained in  $[\Delta_0]$.
For $c\neq 0$, we observe that
$$\uptau(x_3+x_2-cx_1)=(-c)\left(x_4-(-\frac{1}{c})(x_3+x_2)+\left((-\frac{1}{c})^2-\frac{1}{c^2}\right)x_1\right)=(-c)\cdot\ell_{1/c^2}(-\frac{1}{c}).$$
Thus the orbit of $[x_3+x_2-cx_1]$ (for each $c\in\F_q$) is also contained in $[\Delta_0]$.
Moreover, for each $\ell_a(c)\in \Delta_0$, note that $\updelta_{c}^{-1}(\ell_a(c))=\updelta_{c}^{-1}(\updelta_c(\ell_a(0)))=\ell_a(0)\in \Delta_0$ and $ \uptau(\ell_a(c)) =(c^2-a)\cdot \ell_{b}(d)$
where $b=\frac{a}{(c^2-a)^2}$  and $d=\frac{-c}{c^2-a}$. Since $a\in A_0$, it follows that $b\in\A_0$ and
$\ell_{b}(d)\in \Delta_0$. The projective orbit of each $\ell_a(c)\in \Delta_0$ is contained in $[\Delta_0]$.
Therefore, the action  of $\SL_2(\F_q)$ on $\Delta_0$ is projectively stable, as desired. 
\end{proof}

We denote by $\widetilde{\SL}_2(\F_q)$ the subgroup of $\GL_4(\F_q)$ generated by $\upvarphi(\SL_2(\F_q))$ and $\upalpha$. We observe that  
$|\widetilde{\SL}_2(\F_q)|=2\cdot|\upvarphi(\SL_2(\F_q))|=q^3-q$.

\begin{prop}
$\F_q[M_2(\F_q)]^{\widetilde{\SL}_2(\F_q)}$ is a polynomial algebra, generated by $\{f_2,f_3,g_0,g_1\}$.
\end{prop}

\begin{proof}
First of all, we need to check the $\widetilde{\SL}_2(\F_q)$-invariances of these four polynomials. Since they are invariants of $\SL_2(\F_q)$ by Lemma \ref{lem3.1}, it suffices to prove that they are fixed by the action of $\upalpha$. By Proposition \ref{prop2.2}, we see that $f_2$ and $f_3$ both are fixed by $\upalpha$. A direct verification shows that $g_0$ is fixed by  $\upalpha$ and as each $h_a$ is fixed by $\upalpha$, thus $g_1$ is fixed by $\upalpha$ as well. This proves that $f_2,f_3,g_0,g_1$ are $\widetilde{\SL}_2(\F_q)$-invariants. 

Moreover, note that the product of the degrees of $\{f_2,f_3,g_0,g_1\}$ is 
$$1\cdot 2\cdot (q+1)\cdot \frac{q(q-1)}{2}=q^3-q=|\widetilde{\SL}_2(\F_q)|.$$ By  \cite[Corollary 3.1.6]{CW11}, it suffices to prove that $\{f_2,f_3,g_0,g_1\}$ is a homogeneous system of parameters. Let us continue to use the notations in Lemma \ref{lem2.1} and 
suppose $v=(a_1,a_2,a_3,a_4)\in \overline{V}$ denotes an arbitrary point in the algebraic set defined by $\{f_2,f_3,g_0,g_1\}$. Let $\ell_a:=\ell_a(0)=x_4-ax_1$ for all $a\in\A_1$. Since $g_1$ is the projective orbit product of $\{\ell_a\mid a\in\A_1\}$ under the action of $\SL_2(\F_q)$, the assumption that $g_1(v)=0$ implies that there exists one element $\upsigma\in \SL_2(\F_q)$ such that $\upsigma (\ell_a) (v)=\ell_a(\upsigma^{-1}(v))=0$ for some $a$. Thus we may assume that $\ell_a(v)=0$, because the vanishing subvariety of $\{f_2,f_3,g_0,g_1\}$ in $\overline{V}$ is $\SL_2(\F_q)$-stable and we may replace  $v$ with $\upsigma^{-1}(v)$ if necessary. 

Therefore,  $a_4=aa_1.$
The fact that $f_2(v)=f_3(v)=0$ shows that 
$a_3=a_2$ and $a_2^2=a_1a_4=aa_1^2.$  By $g_0(v)=0$, we see that
$2aa_1^{q+1}=2a_2^{q+1}$. Thus $aa_1^{q+1}=a_2^{q+1}=(a_2^2)^{\frac{q+1}{2}}=(aa_1^2)^{\frac{q+1}{2}}=a^{\frac{q+1}{2}}a_1^{q+1}$, which implies that 
$$(a^{\frac{q-1}{2}}-1)a_1^{q+1}=0$$ 
because of $a\neq 0$. Since $a\in\A_1$, it follows that $a^{\frac{q-1}{2}}\neq 1$ and so $a_1^{q+1}=0$ in $\overline{\F}_q$. Hence, $a_1=0$. Together the relations among $\{a_1,a_2,a_3,a_4\}$ obtained above, we see that
$a_2=a_3=a_4=0$. Hence, $v=0$ and \cite[Lemma 2.6.3]{CW11} applies to complete the proof.
\end{proof}

\begin{thm}\label{thm2}
$\F_q[M_2(\F_q)]^{\SL_2(\F_q)}=\F_q[f_2,f_3,g_0,g_1,g_2]$ is a hypersurface.
\end{thm}

\begin{proof}
Since $\F_q[M_2(\F_q)]^{\widetilde{\SL}_2(\F_q)}$ is a Noether normalization of $\F_q[M_2(\F_q)]^{\SL_2(\F_q)}$, it follows from  \cite[Corollary 3.1.4]{CW11}  that $\F_q[M_2(\F_q)]^{\SL_2(\F_q)}$ is a free module over $\F_q[M_2(\F_q)]^{\widetilde{\SL}_2(\F_q)}$ of rank 
$$\frac{|\widetilde{\SL}_2(\F_q)|}{|\upvarphi(\SL_2(\F_q))|}=\frac{q^3-q}{(q^3-q)/2}=2.$$
Thus, we may choose a nonzero $s\in \F_q[M_2(\F_q)]^{\SL_2(\F_q)}\setminus \F_q[M_2(\F_q)]^{\widetilde{\SL}_2(\F_q)}$
 such that $\{1,s\}$ is a basis of $\F_q[M_2(\F_q)]^{\SL_2(\F_q)}$ over $\F_q[M_2(\F_q)]^{\widetilde{\SL}_2(\F_q)}$.
Recall that $U_2(\F_q)$ can be viewed as a Sylow $p$-group of $\SL_2(\F_q)$ and we have proved that $\F_q[M_2(\F_q)]^{U_2(\F_q)}$ is hypersurface, it follows from \cite[Theorem 1]{CHP91} that $\F_q[M_2(\F_q)]^{\SL_2(\F_q)}$ is Cohen-Macaulay.  
Hence, the Hilbert series of $\F_q[M_2(\F_q)]^{\SL_2(\F_q)}$ can be written as
$$\mathcal{H}\left(\F_q[M_2(\F_q)]^{\SL_2(\F_q)}, \uplambda\right)=\frac{1+\uplambda^{\deg(s)}}{(1-\uplambda)(1-\uplambda^2)(1-\uplambda^{q+1})\left(1-\uplambda^{\frac{q(q-1)}{2}}\right)}.$$

By a theorem of Serre (see \cite[Theorem 4.4.3]{BH93}),  the $a$-invariant of $\F_q[M_2(\F_q)]^{\SL_2(\F_q)}$ equals 
$\deg(s)-\left(1+2+(q+1)+\frac{q(q-1)}{2}\right)=\deg(s)-\left(\frac{q(q-1)}{2}+q+4\right)$. Note that $\rank(\uptau\cdot [\updelta_c]-I_4)>1$ for all $c\in\F_q$, thus $\upvarphi(\SL_2(\F_q))$ contains no reflections. It follows from 
\cite[Theorem 4.4]{GJS25} that the $a$-invariants of $\F_q[M_2(\F_q)]^{\SL_2(\F_q)}$ and $\F_q[M_2(\F_q)]$ are equal to each other and the latter is $-4$. Thus
$$\deg(s)=\frac{q(q-1)}{2}+q=\frac{q(q+1)}{2}$$
and the Hilbert series of $\F_q[M_2(\F_q)]^{\SL_2(\F_q)}$ is
$$\mathcal{H}\left(\F_q[M_2(\F_q)]^{\SL_2(\F_q)}, \uplambda\right)=\frac{1+\uplambda^{\frac{q(q+1)}{2}}}{(1-\uplambda)(1-\uplambda^2)(1-\uplambda^{q+1})\left(1-\uplambda^{\frac{q(q-1)}{2}}\right)}.$$

Note that $\deg(g_2)=\frac{q(q-1)}{2}$ and $g_2\notin \F_q[M_2(\F_q)]^{\widetilde{\SL}_2(\F_q)}$ (a direct verification shows that $g_2$ is not fixed by $\upalpha$). A similar argument as in Theorem \ref{thm1} and Corollary \ref{coro1} proves that 
$$\F_q[M_2(\F_q)]^{\SL_2(\F_q)}=\F_q[M_2(\F_q)]^{\widetilde{\SL}_2(\F_q)}\oplus\F_q[M_2(\F_q)]^{\widetilde{\SL}_2(\F_q)}\cdot  g_2=\F_q[f_2,f_3,g_0,g_1,g_2]$$
 is a hypersurface.
\end{proof}

\begin{rem}\label{rem3.4}{\rm
When $q=3$, a magma calculation \cite{BCP97} shows that $\F_3[M_2(\F_3)]^{\SL_2(\F_3)}$ is also a hypersurface generated by 
$\{f_2,f_3,g_0,k_1,k_2\}$, where $\deg(k_1)=\deg(g_1)=3$ and  $\deg(k_2)=\deg(g_2)=6$.  When $q=2$, magma reports  that $\F_2[M_2(\F_2)]^{\SL_2(\F_2)}$ is a hypersurface generated by 
$\{f_2,f_3,g_0,k_1,k_2\}$, where $\deg(k_1)=2$ and  $\deg(k_2)=4$. 
\hbo}\end{rem}

We close this article with a discussion about $\F_q[M_2(\F_q)]^{\SL_2(\F_q)}$ in even characteristic.  The method in odd characteristics above also works for exploring the structure of $\F_q[M_2(\F_q)]^{\SL_2(\F_q)}$. Detailed proofs are similar and will be deleted. We assume that $q=2^s$ for some $s\geqslant 2$ below.  

Note that $\SL_2(\F_q)$ is generated by $U_2(\F_q)$ and $\uptau$ and the action of $\SL_2(\F_q)$ on $M_2(\F_q)$ is faithful. Thus 
$|\upvarphi(\SL_2(\F_q))|= |\SL_2(\F_q)|=q(q^2-1)$ and the subgroup $\widetilde{\SL}_2(\F_q)$ of $\GL_4(\F_q)$ (generated by $\upvarphi(\SL_2(\F_q))$ and $\upalpha$) has order 
$|\widetilde{\SL}_2(\F_q)|=2\cdot|\upvarphi(\SL_2(\F_q))|=2q(q^2-1)$.
Define the following polynomial:
$$k_1:=\prod_{a\in\F_q^\times}h_a.$$
Note that $\deg(k_1)=q(q-1)$. As each nonzero element $a\in\F_q^\times$ is a quadratic residue, thus
$k_1$ is actually a projective orbit product of $\SL_2(\F_q)$ acting on the set of lines in the dual of $M_2(\F_q)$.
Since $q>3$, by the simplicity of $\SL_2(\F_q)$, we see that $k_1$ is an $\SL_2(\F_q)$-invariant. 
Moreover, $k_1$ is also an $\widetilde{\SL}_2(\F_q)$-invariant and
$\F_q[M_2(\F_q)]^{\widetilde{\SL}_2(\F_q)}=\F_q[f_2,f_3,g_0,k_1]$ is a polynomial algebra.

We may write $\F_q[M_2(\F_q)]^{\SL_2(\F_q)}=\F_q[M_2(\F_q)]^{\widetilde{\SL}_2(\F_q)}\oplus \F_q[M_2(\F_q)]^{\widetilde{\SL}_2(\F_q)}\cdot k_2$ for some $\SL_2(\F_q)$-invariant $k_2$ that is not fixed by
$\widetilde{\SL}_2(\F_q)$. By computing the $a$-invariant of $\F_q[M_2(\F_q)]^{\SL_2(\F_q)}$, we observe that $\deg(k_2)=q^2.$  Therefore, there exists an $\SL_2(\F_q)$-invariant $k_2$ of degree $q^2$ such that
$\F_q[M_2(\F_q)]^{\SL_2(\F_q)}=\F_q[f_2,f_3,g_0,k_1,k_2]$ is a hypersurface.

\vspace{5mm}
\noindent \textbf{Acknowledgements}. 
The symbolic computation language MAGMA \cite{BCP97} was very helpful. This research was partially supported by NNSF of China (Grant No. 12561003).
The second-named author would like to thank Professor Runxuan Zhang for her help and encouragement.  
The authors would like to thank the anonymous referees and the editor for their careful reading, constructive comments, and suggestions.

%%%%%%%%%%%%%%%%%%%%%%%%%%References%%%%%%%%%%%%%%%%%%%%%%%%

\raggedright
\end{document}